\documentclass[reqno]{amsart}
\usepackage{amssymb,latexsym}
\usepackage{amsfonts,mathrsfs}
\usepackage[margin=1.4in]{geometry}
\usepackage[all]{xy}

\numberwithin{equation}{section}

\newtheorem{thm}{Theorem}[section]
\newtheorem{lem}[thm]{Lemma}

\newtheorem{prop}[thm]{Proposition}

\theoremstyle{definition}

\newcommand{\ba}{\begin{array}}
\newcommand{\ea}{\end{array}}

\def \qed{\cqfd}

\DeclareMathOperator \dist {dist}

\DeclareMathOperator \supp {supp}

\def\Xint#1{\mathchoice
{\XXint\displaystyle\textstyle{#1}}%
{\XXint\textstyle\scriptstyle{#1}}%
{\XXint\scriptstyle\scriptscriptstyle{#1}}%
{\XXint\scriptscriptstyle\scriptscriptstyle{#1}}%
\!\int}
\def\XXint#1#2#3{{\setbox0=\hbox{$#1{#2#3}{\int}$ }
\vcenter{\hbox{$#2#3$ }}\kern-.58\wd0}}
\def\intav{\Xint-}

\def\qed{\vbox{\hrule
\hbox{\vrule\hbox to 5pt{\vbox to 8pt{\vfil}\hfil}\vrule}\hrule}}

\newcommand{\beg}{\begin{eqnarray*}}
\newcommand{\begn}{\begin{eqnarray}}
\newcommand{\en}{\end{eqnarray*}}
\newcommand{\enn}{\end{eqnarray}}

\begin{document}
\title{some interior regularity estimates for solutions of complex Monge-Amp\`ere equations on a ball}
\subjclass[]{32W20}
\keywords{Complex Monge-Amp\`ere equation, Dirichlet problem on a ball,
 interior
regularity, interior $\mathcal C^{1,\alpha}$ estimate, interior $\mathcal C^{\alpha}$ estimate
}
\author{Chao Li}
\address{School of Mathematical Sciences\\
University of Science and Technology of China\\
Hefei, 230026,P.R. China\\ } \email{leecryst@mail.ustc.edu.cn}
\author{Jiayu Li}
\address{School of Mathematical Sciences\\
University of Science and Technology of China\\
Hefei, 230026\\ and AMSS, CAS, Beijing, 100080, P.R. China\\} \email{jiayuli@ustc.edu.cn}
\author{Xi Zhang}
\address{School of Mathematical Sciences\\
University of Science and Technology of China\\
Hefei, 230026,P.R. China\\ } \email{mathzx@ustc.edu.cn}
\thanks{The authors were supported in part by NSF in
China,  No.11625106, 11571332, 11526212.}

\maketitle

\begin{abstract} In this paper, we consider the Dirichlet problem of  a complex Monge-Amp\`ere equation on a ball in $\mathbb C^n$. With $\mathcal C^{1,\alpha}$ (resp. $\mathcal C^{0,\alpha}$) data, we prove an interior $\mathcal C^{1,\alpha}$ (resp. $\mathcal C^{0,\alpha}$) estimate for the solution. These estimates are generalized versions of the Bedford-Taylor interior $\mathcal C^{1,1}$ estimate.
\end{abstract}

\section{Introduction}

The complex Monge-Amp\`ere equation has many significant applications in complex analysis and complex geometry. In 1970s,  Yau (\cite{Y}) proved the Calabi conjecture by solving a complex complex Monge-Amp\`ere equation on a compact K\"ahler manifold. Since then the  complex Monge-Amp\`ere equation is always a subjuct of intensive studies. Benefiting from the development of theories about the complex Monge-Amp\`ere equation, many problem in complex geometry are solved (see e.g. \cite{Au,CY,MY,Kob,TY1,TY2,TY3,TY4,Ti1,Ti2}).

\medskip

Existence and regularity of solutions are basic objects in the study of the  complex Monge-Amp\`ere equation. Many people contributed a lot to related study (see e.q. \cite{BT,BT1,BT2,CKNS,Ti0,Gb,Ko,GPF,Ko1,Bl1,DZhZh,CH,WY,TWWY}). In \cite{BT}, Bedford and Taylor developed the theory of weak solutions and studied the Dirichlet problem of on a strictly pseudoconvex bounded domain. They proved

\begin{thm}[\cite{BT}, Theorem D]\label{thmBT1}
Let $\Omega$ be a pseudoconvex bounded domain in $\mathbb C^n$. If $0\leq f\in \mathcal C(\bar\Omega)$ and $\varphi\in \mathcal C(\partial\Omega)$, then there exist a unique weak solution $u\in PSH(\Omega)\cap (\bar\Omega)$ to the Dirichlet problem
\begin{equation}
\left\{
\begin{array}{rl}
\det(u_{i\bar j})=f, &\textrm{in }\Omega,\\
u=\varphi, &\textrm{on }\partial\Omega.
\end{array}
\right.
\end{equation}
Furthermore, if $\partial\Omega\in \mathcal C^2$, $\varphi\in \mathcal C^{i,\alpha}(\partial\Omega)$ and $f^{\frac{1}{n}}\in\mathcal C^{0,\frac{i+\alpha}{2}}(\Omega)$ with $i=0,1$ and $\alpha \in (0,1]$, then $u\in \mathcal C^{0,\frac{i+\alpha}{2}}(\bar\Omega)$.
\end{thm}

Bedford and Taylor pointed out that, in the secont part of Theorem \ref{thmBT1}, the result $u\in \mathcal C^{0,\frac{i+\alpha}{2}}(\bar\Omega)$ is optimal according to the H\"older exponent. Even if we assume $\varphi\in \mathcal C^{i,\alpha}(\partial\Omega)$ and $f^{\frac{1}{n}}$ is smooth, generally $u$ doesn't have better global regularity. However, they proved an interior $\mathcal C^{1,1}$ estimate to the solution when $\Omega$ is the unit ball.
\begin{thm}[\cite{BT}, Theorem C] \label{thmBT2} Let B be the unit ball in $\mathbb C^n$. If $u\in PSH(B)\cap \mathcal C(\bar B)$ is the weak solution of the Dirichlet problem
\begin{equation}
\left\{
\begin{array}{rl}
\det(u_{i\bar j})=f, &\textrm{in }B,\\
u=\varphi, &\textrm{on }\partial B,
\end{array}
\right.
\end{equation}
where $\varphi\in \mathcal C^{1,1}(\partial B)$ and $0\leq f^{\frac{1}{n}}\in \mathcal C^{1,1}(\bar{B})$. Then $u\in \mathcal C^{1,1}(B)$.
\end{thm}

The Bedford-Taylor interior $\mathcal C^{1,1}$ estimate has some significant applications. It can be used to study the higher regularity of solutions to the complex Monge-Amp\`ere equation. For example, In Theorem \ref{thmBT2}, if in additionally  $0<f\in \mathcal C^{\infty}(\Omega)$, then based on the  interior $\mathcal C^{1,1}$ estimate, we can prove $u\in \mathcal C^{\infty}(\Omega)$.

\medskip

Bedford and Taylor didn't establish analogous interor $\mathcal C^{1,\alpha}$ or $\mathcal C^{0,\alpha}$ estimates. In the past for a long period of time, it was hard to study local higher regularity of a $\mathcal C^{1,\alpha}$ or $\mathcal C^{0,\alpha}$ solution to a complex Monge-Amp\`ere equation. This might be a reason why the Bedford-Taylor interior $\mathcal C^{1,1}$ estimate was not generalized to $\mathcal C^{1,\alpha}$ or $\mathcal C^{0,\alpha}$ version.

Recently, in \cite{LLZ}, the authors considered the  complex Monge-Amp\`ere equation on a bounded domain $\Omega$ in $\mathbb C^n$
\begin{equation}
\det(u_{i\bar j})=f,
\end{equation}
where $0<f\in \mathcal C^{\alpha}(\Omega)\ (\alpha\in (0,1))$. By using Bedford-Taylor interior $\mathcal C^{1,1}$ estimate (\cite{BT}) and Caffarelli-
Kohn-Nirenberg-Spruck's result (\cite{CKNS}), they proved, if $u\in \mathcal C^{1,\beta}(\Omega)$ with $\beta\in (\beta_0(n,\alpha),1)$, where $\beta_0(n,\alpha)>0$ depend only on $n$ and $\alpha$, then $u\in\mathcal C^{2,\alpha}(\Omega)$. According to this result, we think it is interesting to generalize the Bedford-Taylor interior $\mathcal C^{1,1}$ estimate to the $\mathcal C^{1,\alpha}$ case. In fact, we can prove

\begin{thm} \label{in1a}Let $B_r(0)$ be a ball in $\mathbb{C}^n$. If $u\in PSH(B_r(0))\cap\mathcal{C}(\bar{B}_r(0))$ is the solution to the Dirichlet problem
\begin{equation}\left\{
\begin{array}{rl}
\det(u_{i\bar{j}})=f,&\textrm{in } B_r(0),\\
u=\varphi,&\textrm{on } \partial B_r(0),
\end{array}
\right.
\end{equation}
where $\varphi\in \mathcal{C}^{1,\alpha}(\partial B_r(0))$, $f\geq 0$ and $f^{\frac{1}{n}}\in \mathcal{C}^{1,\alpha}(\bar{B}_r(0))$, $\alpha\in(0,1)$. Then $u\in \mathcal{C}^{1,\alpha}(B_r(0))$. Furthermore, for any $t\in(0,1)$, we have
\begin{equation}[u]_{1,\alpha;B_{(1-t)r}(0)}\leq C(n,\alpha,t)([\varphi]_{1,\alpha;\partial B_r(0)}+r^{1-\alpha}|f^{\frac{1}{n}}|'_{1,\alpha;B_r(0)}),\end{equation}
where the constant $C(n,\alpha,t)$ depend only on $n$, $\alpha$ and $t$.
\end{thm}

In order to  prove Theorem \ref{in1a}, we follow the work of Bedford and Taylor's method and prove  the solution satisfies a second-order difference type inequality (Lemma \ref{klm1}), then use this property to prove the solution is locally $\mathcal C^{1,\alpha}$ continuous (by Lemma \ref{klm2}).

\medskip

For the $\mathcal C^{0,\alpha}\ (\alpha\in (0,1])$ case, we also have the following result

\begin{thm} \label{in0a}Let $B_r(0)$ be a ball in $\mathbb{C}^n$. If $u\in PSH(B_r(0))\cap\mathcal{C}(\bar{B}_r(0))$ is the solution to the Dirichlet problem
\begin{equation}\left\{
\begin{array}{rl}
\det(u_{i\bar{j}})=f,&\textrm{in } B_r(0),\\
u=\varphi,&\textrm{on } \partial B_r(0),
\end{array}
\right.
\end{equation}
where $\varphi\in \mathcal{C}^{0,\alpha}(\partial B_r(0))$, $f\geq 0$ and $f^{\frac{1}{n}}\in \mathcal{C}^{0,\alpha}(\bar{B}_r(0))$, $\alpha\in(0,1]$. Then $u\in \mathcal{C}^{0,\alpha}(B_r(0))$. Furthermore, for any $t\in(0,1)$, we have
\begin{equation}[u]_{0,\alpha;B_{(1-t)r}(0)}\leq C(n,t)([\varphi]_{0,\alpha;\partial B_r(0)}+r^{2-\alpha}|f^{\frac{1}{n}}|'_{0,\alpha;B_r(0)}),\end{equation}
where the constant $C(n,t)$ depend only on $n$ and $t$.
\end{thm}

\medskip

The  Bedford-Taylor interior $\mathcal C^{1,1}$ estimate can be generalized to some special pseudoconvex domain, e.g. polydisks (see \cite{Bl1}). Similarly, Theorem \ref{in1a} and \ref{in0a} can be generalized to these domains. In this paper, we don't go into details.

\medskip

Now we give an overview of this paper. In Section 2, we introduce some notations for  H\"oler (semi-)norms of functions and review the Schauder interior $\mathcal C^{1,\alpha}$ estimate for Poisson equations. In Section 3 and Section 4, we give proof to Theorem \ref{in1a} and \ref{in0a} respectively. In the last section, we introduce some analogous results for complex Monge-Amp\`ere equations on Hermitian manifolds.

\vspace{0.5cm}

\section{Preliminary}

\subsection{Notations for H\"older norms and semi-norms of functions}

For convenience, we first introduce some notations used in \cite{GT}.

Let $\Omega \subset \mathbb{R}^n$ be a bounded open domain. For any $x,y\in\Omega$, we set
\begin{equation*}d_x=\dist(x,\partial\Omega), \qquad d_{x,y}=\min\{d_x,d_y\}.\end{equation*}
For any $u\in \mathcal C^k(\Omega)$ or $u \in  \mathcal{C}^{k,\alpha}(\Omega)$ ($\alpha\in (0,1] $), we define the following quantities:
\begin{equation}\begin{aligned}
&|u|_{0;\Omega}=\sup\limits_{x\in \Omega}|u(x)|,\\
&[u]_{k,0;\Omega}=[u]_{k;\Omega}=\sup\limits_{x\in \Omega}|\nabla^k u(x)|,\\
&[u]_{k,\alpha;x}=\sup\limits_{\begin{subarray}{c} y\in \Omega\\ y\neq x\end{subarray}}\frac{|\nabla^k u(y)-\nabla^k u(x)|}{|y-x|^{\alpha}},\\
&[u]_{k,\alpha;\Omega}=\sup\limits_{x\in\Omega}[u]_{k,\alpha;x}=\sup\limits_{\begin{subarray}{c}x,y\in \Omega\\x\neq y\end{subarray}}\frac{|\nabla^k u(x)-\nabla^k u(y)|}{|x-y|^{\alpha}},\\
\end{aligned}\end{equation}
and
\begin{equation}\begin{aligned}
&[u]^*_{k,0;\Omega}=[u]^*_{k;\Omega}=\sup\limits_{x\in\Omega}d_x^k|\nabla^k\! u(x)|,\\
&|u|^*_{k,0;\Omega}=|u|^*_{k;\Omega}=\sum\limits_{i=0}^{k}[u]^*_{i;\Omega},\\
&[u]^*_{k,\alpha;\Omega}=\sup\limits_{\begin{subarray}{c}x,y\in\Omega\\x\neq y\end{subarray}}d_{x,y}^{k+\alpha} \frac{|\nabla^k\! u(x)-\nabla^k\! u(y)|}{|x-y|^{\alpha}},\\
&|u|^*_{k,\alpha;\Omega}=|u|^*_{k;\Omega}+[u]^*_{k,\alpha;\Omega}.
\end{aligned}\end{equation}

\medskip

By the definitions, we see that
\begin{equation}|u|_{0;\Omega}=[u]_{0;\Omega}=|u|^*_{0;\Omega},\end{equation}
and
\begin{equation}[u]^*_{k-1,1;\Omega}\leq C(n,k)([u]^*_{k-1;\Omega}+[u]^*_{k;\Omega}),\qquad [u]^*_{k;\Omega}\leq C(n,k)[u]^*_{k-1,1;\Omega},\end{equation}
for $k\geq 1$.
When $\Omega$ is convex, we also have
\begin{equation}[u]_{k-1,1;\Omega}\leq C(n,k)[u]_{k;\Omega}, \qquad [u]_{k;B_r}\leq C(n,k)[u]_{k-1,1;\Omega},\end{equation}
for $k\geq 1$.

\medskip

When $\Omega=B_r$ is a ball of radius $r$, we also use the following notations
\begin{equation}\begin{aligned}
&|u|'_{k,0;\Omega}=|u|'_{k;\Omega}=\sum\limits_{i=0}^{k}r^i[u]_{i;\Omega},\\
&|u|'_{k,\alpha;\Omega}=|u|'_{k;\Omega}+r^{k+\alpha}[u]_{k,\alpha;\Omega}.
\end{aligned}\end{equation}

\medskip

For any $f\in \mathcal{C}(\Omega)$ or $f\in \mathcal{C}^{\alpha}(\Omega)\ (\alpha\in (0,1])$, we define
\begin{align}
&|f|^{(k)}_{0,;\Omega}=\sup\limits_{x\in\Omega}d_x^k|f(x)|,\\
&|f|^{(k)}_{0,\alpha;\Omega}=\sup\limits_{x\in\Omega}d_x^k|f(x)|+\sup\limits_{\begin{subarray}{c}x,y\in\Omega\\x\neq y\end{subarray}}d_{x,y}^{k+\alpha} \frac{|f(x)-f(y)|}{|x-y|^{\alpha}}.\end{align}

\medskip

In this paper we will consider the Dirichlet problem on a ball, so we also introduce some notations for H\"older norms and semi-norms of functions defined on sphere.

Let $B_r$ be a balll of radius $r$ in $\mathbb R^n$. For $\varphi\in \mathcal C^{k,\alpha}(\partial B_r)$ with $k=0,1$ and $\alpha\in [0,1]$, we define
\begin{equation}\begin{aligned}&|\varphi|_{0;\partial B_r}=\sup\limits_{x\in\partial B_r}|\varphi(x)|,\\
&[\varphi]_{0,\alpha;\partial B_r}=\sup\limits_{\begin{subarray}{c}x,y\in\partial B_r\\x\neq y\end{subarray}}\frac{|\varphi(x)-\varphi(y)|}{|x-y|^{\alpha}},\qquad \alpha>0,\\
&[\varphi]_{1,\alpha;\partial B_r}=\inf\{[\Phi]_{1,\alpha;B_r}|\Phi\in\mathcal{C}^{1,\alpha}(\bar{B}_r)\ \textrm{and}\ \Phi |_{\partial B_r}=\varphi\}.
\end{aligned}
\end{equation}
For any $\alpha\in (0,1]$, let $\Phi \in \mathcal C^{1,\alpha}(\bar B_r)$ with $\alpha\in (0,1]$. It is easy to check that
\begin{align}
&[\Phi]_{1;B_r}\leq r^{-1}|\Phi|_{0;\partial B_r}+2r^{\alpha}[\Phi]_{1,\alpha;B_r},\\
&|\Phi|_{0;B_r}\leq 2(|\Phi|_{0;\partial B_r}+r^{1+\alpha}[\Phi]_{1,\alpha;B_r}).
\end{align}
\iffalse
Using  the Tayor expansion of $\Phi$, we have
\begin{equation}|\Phi(x)-\Phi(0)-\langle\nabla \Phi(0),x\rangle|\leq |x|^{1+\alpha}[\varphi]_{1,\alpha;\partial B_r},\end{equation}
for any $x\in \bar B_r(0)$. Assume $\nabla \Phi(0)\neq 0$. Let $x_1=r\frac{\nabla\Phi(0)}{|\nabla\Phi(0)|},\ x_2=-x_1$, then $x_1,x_2\in\partial B_r(0)$. We have
\begin{align}&\big{|}\Phi(x_1)-\Phi(0)-r|\nabla \Phi(0)|\big{|}\leq r^{1+\alpha}[\varphi]_{1,\alpha;\partial B_r};\\
&\big{|}\Phi(x_2)-\Phi(0)+r|\nabla \Phi(0)|\big{|}\leq r^{1+\alpha}[\varphi]_{1,\alpha;\partial B_r};\\
\end{align}
and
\begin{equation}|\Phi(x_1)|\leq |\varphi|_{0;\partial B_r},\qquad |\Phi(x_2)|\leq |\varphi|_{0;\partial B_r}.
\end{equation}
By these inequalities, we have
\begin{equation}|\nabla\Phi(0)|\leq r^{-1}|\Phi|_{0;\partial B_r}+r^{1+\alpha}[\Phi]_{1,\alpha;B_r},
\end{equation}
then we can gradually obtain
\begin{align}
&[\Phi]_{1;B_r}\leq r^{-1}|\Phi|_{0;\partial B_r}+2r^{\alpha}[\Phi]_{1,\alpha;B_r},\\
&|\Phi|_{0;B_r}\leq 2(|\Phi|_{0;\partial B_r}+r^{1+\alpha}[\Phi]_{1,\alpha;B_r}).
\end{align}
\fi

\medskip

Given any $\varphi\in \mathcal C^{1,\alpha}(\partial B_r)$, we can find $\{\Phi_k\}_{k=1}^{\infty}\subset \mathcal{C}^{1,\alpha}(B_r)$ such that
\begin{equation}\Phi_k|_{\partial B_r}=\varphi, \qquad \lim\limits_{k\rightarrow \infty}[\Phi_k]_{1,\alpha;B}=[\varphi]_{1,\alpha;\partial B_r}.\end{equation}
Then $\{\Phi_k\}_{k=1}^{\infty}$ is bounded in $\mathcal{C}^{1,\alpha}(\bar B_r)$. Consequently $\{\Phi_k\}_{k=1}^{\infty}$ has a sub-sequence which converges uniformly to some $\Phi_{\infty}\in \mathcal C^{1,\alpha}(\bar B)$, which satisfies
\begin{equation}\Phi_{\infty}|_{\partial B_r}=\varphi, \qquad [\Phi_{\infty}]_{1,\alpha;B_r}=[\varphi]_{1,\alpha;\partial B_r}.\end{equation}
So any $\varphi\in \mathcal C^{1,\alpha}(\partial B_r)$ can be extended to some $\Phi\in \mathcal C^{1,\alpha}(\bar B_r)$ with $[\Phi]_{1,\alpha;B_r}=[\varphi]_{1,\alpha;\partial B_r}$.

\vspace{0.3cm}

\subsection{The Schauder $\mathcal C^{1,\alpha}$ estimate} The following theorem is well-known

\begin{thm}[\cite{GT}, Theorem 3.9]\label{sch1a} Let $\Omega$ be an open set in $\mathbb{R}^n$, and let  $u\in \mathcal{C}^2(\Omega)$ and $f\in \mathcal{C}(\Omega)$ satisfy the Poisson equation $\Delta u=f$ on $\Omega$. Then for any $\alpha\in(0,1)$, we have
\begin{equation}|u|^*_{1,\alpha;\Omega}\leq C(n,\alpha)(|u|_{0;\Omega}+|f|^{(2)}_{0;\Omega}).\end{equation}
\end{thm}

\medskip

We need a refined version of Theorem \ref{sch1a}, which can be seen as a combination of Theorem \ref{sch1a} and Lemma 6.32 of \cite{GT}.

\begin{prop}\label{sch1a'} Let $\Omega$ be an open set in $\mathbb{R}^n$, and let  $u\in \mathcal{C}^2(\Omega)$ and $f\in \mathcal{C}(\Omega)$ satisfy the Poisson $\Delta u=f$ equation on $\Omega$. Then for any $\alpha\in(0,1)$ and $\mu\in (0,1]$, we have
\begin{align}\label{sch1a'1}&[u]^*_{1;\Omega}\leq C(n)(\mu^{-1}|u|_{0;\Omega}+\mu |f|^{(2)}_{0;\Omega}),\\
&\label{sch1a'2}[u]^*_{1,\alpha;\Omega}\leq C(n,\alpha)(\mu^{-1-\alpha}|u|_{0;\Omega}+\mu^{1-\alpha}|f|^{(2)}_{0;\Omega}).
\end{align}
\end{prop}

\begin{proof}First we estimate $[u]^*_{0,1;\Omega}$. Let $x,y$ be two distinct poits in $\Omega$.
\begin{itemize}
\item[a).]$|x-y|> \frac{1}{4}\mu d_{x,y}$. We have
\begin{equation}d_{x,y}\frac{|u(x)-u(y)|}{|x-y|}\leq  4\mu^{-1}(|u(x)|+|u(y)|)\leq 8\mu^{-1}|u|_{0;\Omega}.
\end{equation}
\item[b).]$|x-y|\leq \frac{1}{4}\mu d_{x,y}$. Set $D=B_{\frac{1}{2}\mu d_{x,y}}(x)$. For any $z\in D$, we have
\begin{equation}\textstyle \dist(z, \partial D)\leq \frac{1}{2}\mu d_{x,y}, \qquad d_z\geq d_x-|x-z|\geq \frac{1}{2}d_{x,y}\end{equation}consequently
\begin{equation}\dist(z, \partial D)\leq \mu d_z.\end{equation}
By the definitions of $|f|^{(2)}_{0;D}$ and $|f|^{(2)}_{0;\Omega}$, we have
\begin{equation}|f|^{(2)}_{0;D}\leq \mu^2|f|^{(2)}_{0;\Omega}.\end{equation}
Apply Theorem \ref{sch1a} on $D$
\begin{equation}[u]^*_{1;D}\leq C(n)(|u|_{0;D}+|f|^{(2)}_{0;D})\leq C(n)(|u|_{0;\Omega}+\mu^2|f|^{(2)}_{0;\Omega}).\end{equation}
By the relation between $[u]^*_{1;D}$ and $[u]^*_{0,1;D}$, we have
\begin{equation}[u]^*_{0,1;D}\leq C(|u|_{0;D}+|f|^{(2)}_{0;D}).\end{equation}
By the definition of $[u]^*_{0,1;D}$, we have
\begin{equation}\min\{\dist(x,\partial D),\dist(y,\partial D)\}\frac{|u(x)-u(y)|}{|x-y|}\leq [u]^*_{0,1;D},\end{equation}
where
\begin{equation}\textstyle
\min\{\dist(x,\partial D),\dist(y,\partial D)\}=\frac{1}{2}\mu d_{x,y}-|x-y|\geq \frac{1}{4}\mu d_{x,y}.
\end{equation}
Then we obtained
\begin{equation}d_{x,y}\frac{|u(x)-u(y)|}{|x-y|}\leq C(n)(\mu^{-1}|u|_{0;\Omega}+\mu |f|^{(2)}_{0;\Omega}).
\end{equation}
\end{itemize}
Combine a) and b), we have
\begin{equation}[u]^*_{0,1;\Omega}\leq C(n)(\mu^{-1}|u|_{0;\Omega}+\mu |f|^{(2)}_{0;\Omega}),\end{equation}
Since $[u]^*_{1;\Omega}\leq [u]^*_{0,1;\Omega}$, we complete the proof of (\ref{sch1a'1}).

\medskip

By similar discussion, we can obtain
\begin{equation}[u]^*_{1,\alpha;\Omega}\leq C(n,\alpha)(\mu^{-1-\alpha}|u|_{0;\Omega}+\mu^{-\alpha}[u]^*_{1;\Omega}+\mu^{1-\alpha}|f|^{(2)}_{0;\Omega}).\end{equation}
Substitute (\ref{sch1a'1}) into this inequlity, then we obtain (\ref{sch1a'2}).
\end{proof}

\vspace{0.5cm}

\section{Proof of Theorem \ref{in1a}}

In this section we denote $B_r=B_r(0)\subset \mathbb C^{n}$ and  $B=B_1$.

\subsection{Some tricks to simplify the proof} We need the following three usual tricks to simplify the  proof of our interior $\mathcal C^{1,\alpha}$ estimate.

First, when proving Theorem \ref{in1a}, we can consider the case  $r=1$ only. For the case  $r\neq 1$, we can consider $\tilde u(z)=r^2u\left(\frac{z}{r}\right)$ instead.

\medskip

Second, when proving Theorem \ref{in1a}, we can assume $\varphi\in \mathcal C^{1,\alpha}(\bar B_r)$ and satisfies
\begin{equation}|\varphi|_{0;B_r}\leq r^{1+\alpha}[\varphi]_{1,\alpha;\partial B_r},\qquad [\varphi]_{1;B_r}\leq r^{\alpha}[\varphi]_{1,\alpha;\partial B_r}, \qquad [\varphi]_{1,\alpha;B_r}=[\varphi]_{1,\alpha;\partial B_r}.\end{equation}
In section 2, We have pointed out that $\varphi$ can be extended to some $\Phi\in \mathcal C^{1,\alpha}(\bar B_r)$ with $[\Phi]_{1,\alpha;B_r}=[\varphi]_{1,\alpha;\partial B_r}$. To show this trick is reasonable, we only need to transform the equation.
We set
\begin{align}&\tilde{\varphi}(x)=\Phi(x)-\Phi(0)-\langle\nabla \Phi(0), x\rangle,\\
&\tilde{u}(x)=u(x)-\Phi(0)-\langle\nabla \Phi(0), x\rangle,
\end{align}
where $\nabla \Phi(0)$ and $x\in B_r$ are treated as real vector of dimension $2n$, $\langle\nabla \Phi(0), x\rangle$ is their inner product.
Then $\tilde{u}\in PSH(B_r)\cap \mathcal{C}(\bar B_r)$ and satisfies
\begin{equation}\left\{
\begin{array}{rl}
\det(\tilde u_{i\bar{j}})=f,&\textrm{in } B,\\
\tilde u=\tilde \varphi,&\textrm{on } \partial B,
\end{array}
\right.
\end{equation}
Applying Taylor expansion, it is easy to check that
\begin{equation}|\tilde\varphi|_{0;B_r}\leq r^{1+\alpha}[\varphi]_{1,\alpha;\partial B_r},\qquad [\tilde\varphi]_{1;B_r}\leq r^{\alpha}[\varphi]_{1,\alpha;\partial B_r}, \qquad [\tilde\varphi]_{1,\alpha;B_r}=[\varphi]_{1,\alpha;\partial B_r}.\end{equation}
$\tilde u$ differs from $u$ by only a linear function, so their have same regularity. Furthermore, for any relative compact open subset $D$ in $B_r$, if either $\tilde u$ or $u$ is in $\mathcal C^{1,\alpha}(\bar D)$, then
\begin{equation}[u]_{1,\alpha;D}=[\tilde u]_{1,\alpha;D}.\end{equation}
By these conclusion, the second trick really works.

At the other hand, for the mentioned $\tilde u$, we have
\begin{equation}u(x+h)+u(x-h)-2u(x)=\tilde u(x+h)+\tilde u(x-h)-2\tilde u(x),\end{equation}
whenever the two second-order differences is well defined. So the second trick can be used when we try to estimate such a second-order difference, i.e. to prove Lemma \ref{klm1}.

\medskip

Third, when proving the inequalities in Theorem \ref{in1a} and  Lemma \ref{klm1}, we can assume  $u$ is $\mathcal C^2$. Assume $r=1$ and $\varphi\in \mathcal C^{1,\alpha}(\bar B)$ by expanding. For small $\varepsilon>0$, we define following functions on $\bar B$
\begin{align}&\varphi_{\varepsilon}(x)=\int_{\mathbb C^n}\rho_{\varepsilon}(y)\varphi\left((1-\varepsilon)(x-y)\right),\\
&{f_{\varepsilon}}^{\frac{1}{n}}(x)=\int_{\mathbb C^n}\rho_{\varepsilon}(y)f^{\frac{1}{n}}\left((1-\varepsilon)(x-y)\right).
\end{align}
Then $\varphi_{\varepsilon}, {f_{\varepsilon}}^{\frac{1}{n}}\in \mathcal C^{\infty}(\bar B)$ and
\begin{equation}[\varphi_{\varepsilon}]_{1,\alpha;B}\leq [\varphi]_{1,\alpha;B},\qquad [{f_{\varepsilon}}^{\frac{1}{n}}]_{1,\alpha;B}\leq [f^{\frac{1}{n}}]_{1,\alpha;B}.\end{equation}
Let $u_{\varepsilon}\in PSH(B)\cap \mathcal C(\bar B)$ solve
\begin{equation}\left\{
\begin{array}{rl}
\det((u_{\varepsilon})_{i\bar{j}})=f_{\varepsilon},&\textrm{in } B,\\
u_{\varepsilon}=\varphi_{\varepsilon},&\textrm{on } \partial B,
\end{array}
\right.
\end{equation}
Then $u$ is smooth by \cite{CKNS}. When $\varepsilon\rightarrow 0$,  $\varphi_{\varepsilon}$ and ${f_{\varepsilon}}$ converge  uniformly to $\varphi$ and $f$ respectively, by the comparison principle one can easily prove that $u_{\varepsilon}$ converges uniformly to $u$. If the interior estimate in Theorem \ref{in1a} (or Lemma \ref{klm1}) is true for all $u_{\varepsilon}$, then it is also true for $u$. By this approximation of smooth solutions, we can assume $u$ itself is smooth.

One can try to compute in the language of currents. In this way, the condition that $u$ is continuous is enough, but some formulas will become quite complicated.

\vspace{0.3cm}

\subsection{A second-order difference type inequality}

We have the following lemma
\begin{lem}\label{klm1} If $u\in PSH(B)\cap\mathcal{C}(\bar{B})$ is the solution of the equation
\begin{equation}\left\{
\begin{array}{rl}
\det(u_{i\bar{j}})=f,&\textrm{in } B,\\
u=\varphi,&\textrm{on } \partial B,
\end{array}
\right.
\end{equation}
where, $\varphi\in \mathcal{C}^{1,\alpha}(\partial B)$, $f\geq 0$ and $f^{\frac{1}{n}}\in \mathcal{C}^{1,\alpha}(\bar{B})$, $\alpha\in(0,1]$. Then for any $t\in(0,1],\ x\in B_{1-t},\ h\in B_{\frac{1}{2}t}$, we have
\begin{equation}u(x+h)+u(x-h)-2u(x)\leq C(n,t)([\varphi]_{1,\alpha;\partial B}+|f^{\frac{1}{n}}|'_{1,\alpha;B})|h|^{1+\alpha}.
\end{equation}
\end{lem}

Lemma \ref{klm1} is a generalization of Proposition 6.6 of \cite{BT}. Our proof is also similar to the origin one.

We need to use some automorphism of $B$.  For any $a\in B$, let $v(a)=\sqrt{1-|a|^2}$. We define $T_a\in Auto(B)$ as follow
\begin{equation}\label{Ta}
\begin{aligned}&T_a(z)=\Gamma_a\frac{z-a}{1-a^*z},\qquad a\neq 0,\\
&T_0(z)=z,
\end{aligned}
\end{equation}
where
\begin{equation}\Gamma_a=\frac{aa^*}{1-v(a)}-v(a)I.\end{equation}
We treat $a$ and $z$ as $n\times1$ matrices, the upper index $*$ refer to the transposed conjugation, so $\Gamma_a$ is an $n\times n$ matrix.

It is easy to verify that $T:B\times \bar{B}\rightarrow \bar{B}$
\begin{equation*}T(a,z)=T_a(z)\end{equation*}
is a smooth map. Furthermore
\begin{equation}\label{Ta0}T_a(a)=0,\qquad T_a(0)=-a.\end{equation}

\begin{proof}[Proof of Lemma \ref{klm1}] By the second trick mentioned in 3.1,  we can assume $\varphi\in \mathcal{C}^{1,\alpha}(\bar B)$ and
\begin{equation}|\varphi|_{0;B}\leq [\varphi]_{1,\alpha;\partial B},\qquad [\varphi]_{1;B}\leq [\varphi]_{1,\alpha;\partial B}, \qquad [\varphi]_{1,\alpha;B}=[\varphi]_{1,\alpha;\partial B}.\end{equation}

\medskip

Consider the following functions defined on $B\times \bar B$
\begin{align}\label{U}&U(x,z)=u_x(z)=u(T_{-x}(z)),\\
\label{F}&F(x,z)=f_x(z)=f(T_{-x}(z))|\det JT_{-x}(z)|^2,\\
\label{Phi}&\Phi(x,z)=\varphi_x(z)=\varphi(T_{-x}(z)),
\end{align}
where $JT_a\ (a \in B)$ is the complex Jacobian matrix of the holomorphic map $T_a$.
For any fixed $x\in B$, $u_x$ is in $PSH(B)\cap \mathcal{C}(\bar B)$ and satisfies
\begin{equation*}\left\{
\begin{array}{rl}
\det((u_x)_{i\bar{j}})=f_x,&\textrm{in } B,\\
u_x=\varphi_x,&\textrm{on } \partial B.
\end{array}
\right.
\end{equation*}
Furthermore
\begin{equation*}U(x,0)=u_x(0)=u(x).\end{equation*}

\medskip

Since $\varphi\in \mathcal{C}^{1,\alpha}(\bar B)$, $\Phi\in C^{1,\alpha}(\bar B_{1-\frac{1}{2}t}\times \bar B)$. At the same time, for any $z\in \bar B$, we have the following uniform estimates
\begin{equation}[\Phi(\cdot,z)]_{1,\alpha;B_{1-\frac{1}{2}t}}\leq C(n,t)([\varphi]_{1;B}+[\varphi]_{1,\alpha;B}).\end{equation}
Similarly, $F^{\frac{1}{n}}\in C^{1,\alpha}(\bar B_{1-\frac{1}{2}t}\times \bar B)$. For any $z\in \bar B$, we have the following uniform estimates
\begin{equation}[F^{\frac{1}{n}}(\cdot,z)]_{1,\alpha;B_{1-\frac{1}{2}t}}\leq C(n,t)(|f^{\frac{1}{n}}|_{0;B}+[f^{\frac{1}{n}}]_{1;B}+[f^{\frac{1}{n}}]_{1,\alpha;B}).\end{equation}
For any fixed $x\in B_{1-t},\ h\in B_{\frac{1}{2}t}$, by the Taylor expansion we can obtain
\begin{align}\Phi(x+h,z)+\Phi(x-h,z)-2\Phi(x,z)&\leq 2[\Phi(\cdot,z)]_{1,\alpha;B_{1-\frac{1}{2}t}}|h|^{1+\alpha},\\
F^{\frac{1}{n}}(x+h,z)+F^{\frac{1}{n}}(x-h,z)-2F^{\frac{1}{n}}(x,z)&\geq 2[F^{\frac{1}{n}}(\cdot,z)]_{1,\alpha;B_{1-\frac{1}{2}t}}|h|^{1+\alpha}.
\end{align}
When $z\in \partial B$, $U(\cdot,z)=\Phi(\cdot,z)$, so we have
\begin{equation}(u_{x+h}+u_{x-h}-2u_x)|_{\partial B}\leq 2[\Phi(\cdot,z)]_{1,\alpha;B_{1-\frac{1}{2}t}}|h|^{1+\alpha}.\end{equation}
Consider
\begin{equation}W(z)=u_{x+h}(z)+u_{x-h}(z)+|h|^{1+\alpha}(-A_1+A_2(|z|^2-1)),\end{equation}
where
\begin{equation}A_1=2[\Phi(\cdot,z)]_{1,\alpha;B_{1-\frac{1}{2}t}},\qquad A_2=2[F^{\frac{1}{n}}(\cdot,z)]_{1,\alpha;B_{1-\frac{1}{2}t}},\end{equation}
Then $W\in PSH(B)\cap \mathcal{C}(\bar B)$, and
\begin{equation}W|_{\partial B}=(u_{x+h}+u_{x-h})|_{\partial B}-A_1|h|^{1+\alpha}\leq 2u_x|_{\partial B}\end{equation}
Furthermore, on $B$
\begin{equation}\begin{aligned}
\det(W_{i\bar j})^{\frac{1}{n}}&=\det((u_{x+h})_{i\bar j}+(u_{x-h})_{i\bar j}+A_2|h|^{1+\alpha}\delta_{ij})^{\frac{1}{n}}\\
&\geq\det((u_{x+h})_{i\bar j})^{\frac{1}{n}}+\det((u_{x-h})_{i\bar j})^{\frac{1}{n}}+A_2|h|^{1+\alpha}\\
&=f_{x+h}^{\frac{1}{n}}+f_{x-h}^{\frac{1}{n}}+A_2|h|^{1+\alpha}\\
&\geq 2f_x^{\frac{1}{n}}=\det((2u_x)_{i\bar j})^{\frac{1}{n}}.
\end{aligned}
\end{equation}
By the comparison principle (\cite{BT}, Theorem A), $W\leq 2u_x$. Consequently
\begin{equation}u(x+h)+u(x-h)-2u(x)=W(0)-2u_x(0)+(A_1+A_2)|h|^{1+\alpha}\leq (A_1+A_2)|h|^{1+\alpha},\end{equation}
By the expression of $A_1$ and $A_2$, we conclude the proof of Lemma \ref{klm1}.
\end{proof}

\vspace{0.3cm}

\subsection{Mean value inequality and H\"older continuity} Let $\Omega$ be a bounded open set in $\mathbb R^{n}$. If $v\in \mathcal C^{1,\alpha}(\Omega)$, then for any $x\in \Omega$ and any positive $h\leq \frac{1}{2}d_x$, we have
\begin{equation}\label{mvineq0}\left|\intav_{\partial B_h(x)}v(y)d\sigma_y-v(x)\right|\leq C_xh^{1+\alpha},\end{equation}
where we can take the constant $C_x$ to be $[u]_{1,\alpha;B_{\frac{1}{2}d_x}(x)}$.

Inversely, let $v\in \mathcal C(\Omega)$ and $C_x$ (treated as a function of $x\in \Omega$) be locally bounded.  If (\ref{mvineq0}) holds for any $x\in \Omega$ and any positive $h\leq \frac{1}{2}d_x$, then $v\in \mathcal C^{1,\alpha}(\Omega)$. This result is a corollary of the following Lemma

\begin{lem}\label{klm2} Let  $r\geq t>0$, $B_{r+t}(0)$ be a ball in $\mathbb{R}^n$, and $v\in \mathcal{C}(B_{t+r}(0))$. If there exist a constant $A$ such that
\begin{equation}\label{ineqmv1a}\left|\intav_{\partial B_h(0)}v(x+y)d\sigma_y-v(x)\right|\leq Ah^{1+\alpha},\qquad \forall x\in B_r, h\in(0,t].\end{equation}
Then $v\in \mathcal C^{1,\alpha}(B_r(0))$. Furthermore
\begin{equation}[v]^*_{1,\alpha;B_r}\leq C(n,\alpha)(r^{1+\alpha}t^{-1-\alpha}|v|_{0;B_{r+t}(0)}+Ar^{1+\alpha}).\end{equation}\end{lem}

\begin{proof}Let $\rho$ be a  radially symmetrical function on $\mathbb R^n$ satisfying
\begin{itemize}
\item[1).] $\rho\geq 0$ and $\supp \rho\subset B_1(0)$;
\item[2).] $\int_{B_1(0)}\rho=1$.
\end{itemize}
By choosing proper $\rho$, $\sup |\nabla^k \rho|\ (k=0,1,\cdots)$ can be seen as constants depending only on $n$ and $k$.

For any $\varepsilon >0$, we define
\begin{equation}\rho_{\varepsilon}(x)=\varepsilon^{-n}\rho\left(\frac{x}{\varepsilon}\right).\end{equation}

For any $h\in(0,t]$, we set
\begin{equation}v_h=\rho_h*v.\end{equation}
Then $v_h\in \mathcal{C}^{\infty}(B_{r+t-h}(0))$.

For any $x\in B_r(0)$, we have
\begin{equation}\Delta v_h(x)=\int_{B_h(0)}(\Delta \rho_h)(y)v(x-y)dy=\int_0^hds\int_{\partial B_s(0)}(\Delta \rho_h)(y)v(x-y)d\sigma_y.\end{equation}
By the definition of $\rho_h$, $\Delta \rho_h$ is a radially symmetric function, we can treat it as a function of the radius
\begin{equation}(\Delta \rho_h)(x)=(\Delta \rho_h)(|x|), \end{equation}
Furthermore, we have
\begin{equation}|\Delta \rho_h|\leq C(n)h^{-n-2}.\end{equation}
At the same time
\begin{equation}0=\int_{B_h(0)}(\Delta \rho_h)(y)v(x)dy=\int_0^hds\int_{\partial B_s(0)}(\Delta \rho_h)(y)v(x)d\sigma_y,\end{equation}
consequently
\begin{equation}\begin{aligned}\Delta v_h(x)&=\int_0^h(\Delta \rho_h)(s)ds\int_{\partial B_s(0)}(v(x-y)-v(x))d\sigma_y\\
&=\int_0^h(\Delta \rho_h)(s)|\partial B_s(0)|\left(\intav_{\partial B_s(0)}v(x-y)d\sigma_y-v(x)\right)ds.
\end{aligned}
\end{equation}
Combine this equality with (\ref{ineqmv1a}), we have
\begin{equation}|\Delta v_h(x)|\leq \int_0^h C(n)h^{-n-2}s^{n-1}As^{1+\alpha}ds\leq C(n)Ah^{\alpha-1}.
\end{equation}
By similar discussion, we can obtain
\begin{equation}|v_h(x)-v(x)|\leq C(n)Ah^{1+\alpha}.\end{equation}
Then we have
\begin{equation}\label{mv1a1}|v_h-v|_{0;B_r(0)}\leq C(n)Ah^{1+\alpha},\qquad |\Delta v_h|_{0;B_r(0)}\leq C(n)Ah^{\alpha-1}.\end{equation}

For $k=0,1,2,\cdots$, we denote $h_k=2^{-k}t$ and define
\begin{equation}w_k=v_{h_{k+1}}-v_{h_k},\end{equation}
then we have
\begin{equation}|w_k|_{0;B_r(0)}\leq C(n)Ah_k^{1+\alpha},\qquad |\Delta w_k|_{0;B_r(0)}\leq C(n)Ah_k^{\alpha-1}.\end{equation}
By Proposition \ref{sch1a'}, for any $\gamma\in(0,1)$and $\mu\in (0,1]$, we have
\begin{align}&[\nabla w_k]^*_{1;B_r(0)}\leq C(n)(\mu^{-1}|w_k|_{0;B_r(0)}+\mu |\Delta w_k|^{(2)}_{0;B_r(0)}),\\
&|w_k|^*_{1,\gamma;B_r(0)}\leq C(n,\gamma) (\mu^{-1-\gamma}|w_k|_{0;B_r(0)}+\mu^{1-\gamma}|\Delta w_k|^{(2)}_{0;B_r(0)}),\end{align}
St $ \mu=2^{-k}\frac{t}{r}$, then we have
\begin{align}\label{mv1a2}&[w_k]^*_{1;B_r(0)}\leq C(n)Art^{\alpha}2^{-\alpha k},\\
\label{mv1a3}&[w_k]^*_{1,\gamma;B_r(0)}\leq C(n,\gamma)Ar^{1+\gamma}t^{\alpha-\gamma}2^{(\gamma-\alpha)k}.\end{align}

\medskip

First , we need to show that for any $\beta\in (0,\alpha)$, $v\in \mathcal C^{1,\beta}(B_r(0))$. By (\ref{mv1a1}), (\ref{mv1a2}) and (\ref{mv1a3}), we have
\begin{equation}|w_k|^*_{1,\beta;B_r(0)}\leq C(n,\beta)Ar^{1+\beta}t^{\alpha-\beta}2^{-(\alpha-\beta)k}.\end{equation}
Notice that $v_{h_k}=v_t+\sum\limits_{i=0}^{k-1}w_i$, so we have
\begin{equation}\begin{aligned}
|v_{h_k}|^*_{1,\beta;B_r(0)}&\leq |v_t|^*_{1,\beta;B_r(0)}+\sum\limits_{i=0}^{k-1}|w_i|^*_{1,\beta;B_r(0)}\\
&\leq |v_t|^*_{1,\beta;B_r(0)}+C(n,\beta)Ar^{1+\beta}t^{\alpha-\beta}\sum\limits_{i=0}^{k-1}2^{-(\alpha-\beta)i}\\
&\leq |v_t|^*_{1,\beta;B_r(0)}+C(n,\alpha,\beta)Ar^{1+\beta}t^{\alpha-\beta}.
\end{aligned}\end{equation}
This implies $|v_{h_k}|^*_{1,\beta;B_r(0)}$ is uniformly bounded. Furthermore, by (\ref{mv1a1}), when $k\rightarrow \infty$, $v_{h_k}$ uniformmaly converges to $v$. So $v\in \mathcal{C}^{1,\beta}(B_r(0))$, and $|v|^*_{1,\beta;B_r(0)}$ is bounded.

\medskip

Next, we need to estimate $[v_{h_k}-v]^*_{1;B_r(0)}$ and $[v_{h_k}]_{1,\gamma;B_r(0)}$,  where $\gamma\in (\alpha, 1)$. Clearly $v-v_{h_k}=\sum\limits_{i=k}^{\infty}w_i$ and $v_{h_k}=v_t+\sum\limits_{i=0}^{k-1}w_i$.  By (\ref{mv1a2}) and (\ref{mv1a3}), we obtained
\begin{equation}\begin{aligned}
\label{mv1a4}[v-v_{h_k}]^*_{1;B_r(0)}&\leq\sum\limits_{i=k}^{\infty}[w_i]_{1;B_r(0)}\leq C(n)Art^{\alpha}\sum\limits_{i=k}^{\infty}2^{-\alpha i}\\
&\leq C(n,\alpha)Art^{\alpha}2^{-\alpha k},
\end{aligned}
\end{equation}
and
\begin{equation}\begin{aligned}
\label{mv1a5}[v_{h_k}]^*_{1,\gamma;B_r(0)}&\leq [v_t]^*_{1,\gamma;B_r(0)}+\sum\limits_{i=0}^{k-1}[w_i]^*_{1,\gamma;B_r(0)}\\
&\leq [v_t]^*_{1,\gamma;B_r(0)}+C(n,\gamma)Ar^{1+\gamma}t^{\alpha-\gamma}\sum\limits_{i=0}^{k-1}2^{(\gamma-\alpha)i}\\
&\leq [v_t]^*_{1,\gamma;B_r(0)}+C(n,\gamma,\alpha)Ar^{1+\gamma}t^{\alpha-\gamma}2^{(\gamma-\alpha)k}.
\end{aligned}\end{equation}

\medskip

At last, we start to estimate $[v]^*_{1,\alpha;B_r(0)}$. Let $x,y$ be any two distinct points in $B_r(0)$, denote
\begin{equation}
d_x=\dist(x,\partial B_r(0)), \qquad d_y=\dist(y,\partial B_r(0)), \qquad d_{x,y}=\min\{d_x,d_y\}.\end{equation}
\begin{itemize}
\item[a).]$|x-y|\geq\frac{t}{r}d_{x,y}$. In this case
\begin{equation}\begin{aligned}
d_{x,y}^{1+\alpha}\frac{|\nabla v(x)-\nabla v(y)|}{|x-y|^{\alpha}}&\leq \left(\frac{|x-y|}{d_{x,y}}\right)^{\alpha}(d_x|\nabla v(x)|+d_y|\nabla v(y)|)\\
&\leq 2r^{\alpha}t^{-\alpha}[v]^*_{1;B_r(0)}.
\end{aligned}\end{equation}
on the other hand
\begin{equation}[v]^*_{1;B_r(0)}\leq [v_t]^*_{1;B_r(0)}+[v \!-\! v_t]^*_{1;B_r(0)}\leq [v_t]^*_{1;B_r(0)}+C(n,\alpha)Art^{\alpha},
\end{equation}
so we have
\begin{equation}d_{x,y}^{1+\alpha}\frac{|\nabla v(x)-\nabla v(y)|}{|x-y|^{\alpha}}\leq 2r^{\alpha}t^{-\alpha}[v_t]^*_{1;B_r(0)}+C(n,\alpha)Ar^{1+\alpha}.\end{equation}
\item[b).]$|x-y|<\frac{t}{r}d_{x,y}$. We choose integer $k\geq 1$ such that
\begin{equation}\label{mv1a*}
2^{-k}\frac{t}{r}\leq \frac{|x-y|}{d_{x,y}}< 2^{-k+1}\frac{t}{r}.
\end{equation}
 Useing (\ref{mv1a4}) and (\ref{mv1a5}), we obtain
\begin{align}&d_x|\nabla v(x)-\nabla v_{h_k}(x)|\leq C(n)Art^{\alpha}2^{-\alpha k},\\
&d_y|\nabla v(y)-\nabla v_{h_k}(y)|\leq C(n)Art^{\alpha}2^{-\alpha k},
\end{align}
and
\begin{equation}
d_{x,y}^{1+\gamma}\frac{|\nabla v_{h_k}(x) - \nabla v_{h_k}(y)|}{|x-y|^{\gamma}}\\
\leq [v_t]^*_{1,\gamma;B_r(0)}+C(n,\gamma,\alpha)Ar^{1+\gamma}t^{\alpha-\gamma}2^{(\gamma-\alpha)k}.
\end{equation}
Furthermore, we have
\begin{equation}\begin{aligned}
d_{x,y}^{1+\alpha}|\nabla v(x)-\nabla v(y)|&\leq d_{x,y}^{1+\alpha}(|\nabla v(x)-\nabla v_{h_k}(x)|+|\nabla v(y)-\nabla v_{h_k}(y)|\\
&\quad+|\nabla v_{h_k}(x)-\nabla v_{h_k}(y)|)\\
&\leq d_{x,y}^{\alpha}(d_x|\nabla v(x)-\nabla v_{h_k}(x)|+d_y|\nabla v(y)-\nabla v_{h_k}(y)|)\\
&\quad +d_{x,y}^{1+\alpha}|\nabla v_{h_k}(x) - \nabla v_{h_k}(y)|,
\end{aligned}\end{equation}
consequently
\begin{equation}\begin{aligned}
&\quad d_{x,y}^{1+\alpha}\frac{|\nabla v(x) - \nabla v(y)|}{|x-y|^{\alpha}}\\
&\leq \left(\frac{|x-y|}{d_{x,y}}\right)^{\gamma-\alpha}([v_t]^*_{1,\gamma;B_r(0)}+C(n,\gamma,\alpha)Ar^{1+\gamma}t^{\alpha-\gamma}2^{(\gamma-\alpha)k})\\
&\quad +C(n,\gamma,\alpha)Art^{\alpha}\left(\frac{|x-y|}{d_{x,y}}\right)^{-\alpha}2^{-\alpha k}\\
&\leq 2r^{\alpha-\gamma}t^{\gamma-\alpha}[v_t]^*_{1,\gamma;B_r(0)}+C(n,\gamma,\alpha)Ar^{1+\alpha}.
\end{aligned}\end{equation}
\end{itemize}

Combine a) and b), we have
\begin{equation}d_{x,y}^{1+\alpha}\frac{|\nabla v(x)-\nabla v(y)|}{|x-y|^{\alpha}}\leq 2r^{\alpha}t^{-\alpha}[v_t]^*_{1;B_r(0)}+2r^{\alpha-\gamma}t^{\gamma-\alpha}[v_t]^*_{1,\gamma;B_r(0)}+C(n,\gamma,\alpha)Ar^{1+\alpha}.\end{equation}
By the definition of $[v]^*_{1,\alpha;B_r(0)}$, we have
\begin{equation}\label{mv1a6}[v]^*_{1,\alpha;B_r(0)}\leq 2r^{\alpha}t^{-\alpha}[v_t]^*_{1;B_r(0)}+2r^{\alpha-\gamma}t^{\gamma-\alpha}[v_t]^*_{1,\gamma;B_r(0)}+C(n,\gamma,\alpha)Ar^{1+\alpha}.\end{equation}
For $v_t$, by $|v_t|_{0;B_r(0)}\leq |v|_{0;B_{r+t}(0)}$ and $|\Delta v_t|_{0;B_r(0)}\leq C(n)At^{\alpha-1}$, we have the following estimate
\begin{align}
&[v_t]^*_{1;B_r(0)}\leq C(n)(rt^{-1}|v|_{0;B_{r+t}(0)}+Art^{\alpha}),\\
&[v_t]^*_{1,\gamma;B_r(0)}\leq C(n,\gamma)(r^{1+\gamma}t^{-1-\gamma}|v|_{0;B_{r+t}(0)}+Ar^{1+\gamma} t^{\alpha-\gamma}),
\end{align}
Substitute these estimates into (\ref{mv1a6}) and set $\gamma=\frac{1+\alpha}{2}$, we obtain
\begin{equation*}[v]^*_{1,\alpha;B_r(0)}\leq C(n,\alpha)(r^{1+\alpha}t^{-1-\alpha}|v|_{0;B_{r+t}(0)}+Ar^{1+\alpha}).\end{equation*}
This concludes the proof of Lemma \ref{klm2}.
\end{proof}

\vspace{0.3cm}

\subsection{Proof of Theorem \ref{in1a}} Using Lemma \ref{klm1}, Lemma \ref{klm2} and the fact that $u$ is plurisubharmonic, we can easily prove Theorem \ref{in1a}.
\begin{proof}We only need to consider the case $r=1$. We assume $\varphi\in \mathcal{C}^{1,\alpha}(\bar B)$ and
\begin{equation}|\varphi|_{0;B}\leq [\varphi]_{1,\alpha;\partial B},\qquad [\varphi]_{1;B}\leq [\varphi]_{1,\alpha;\partial B}, \qquad [\varphi]_{1,\alpha;B}=[\varphi]_{1,\alpha;\partial B}.\end{equation}
By comparison principle we can easily obtain a estimate for  $[u]_{0;B}$
\begin{equation}\label{in1a_0}|u|_{0,\alpha;B}\leq |\varphi|_{0;B}+|f^{\frac{1}{n}}|_{0;B}\leq [\varphi]_{1,\alpha;\partial B}+|f^{\frac{1}{n}}|_{0;B}.\end{equation}

\medskip

By Lemma \ref{klm1}, for any $x\in B_{1-\frac{1}{2}t},\ h\in B_{\frac{1}{4}t}$, we have
\begin{equation}\label{in1a_1}u(x+h)+u(x-h)-2u(x)\leq A|h|^{1+\alpha}.
\end{equation}
where
\begin{equation}\label{in1a_2} A=C(n,t)([\varphi]_{1,\alpha;\partial B}+|f^{\frac{1}{n}}|'_{1,\alpha;B}).\end{equation}
For any fixed $x\in B_{1-\frac{1}{2}t},\ h\in(0,\frac{1}{4}t)$, obviously
\begin{equation}\intav_{\partial B_h(x)}u(x+y)d\sigma_y=\intav_{\partial B_h(x)}u(x-y)d\sigma_y,\end{equation}
so we have
\begin{equation}\label{in1a_3}\intav_{\partial B_h(x)}u(x+y)d\sigma_y-u(x)=\frac{1}{2}\intav_{\partial B_h(x)}(u(x+y)+u(x-y)-2u(x))d\sigma_y.\end{equation}
$u$ is plurisubharmonic, so it is subharmonic. By this property, (\ref{in1a_1}) and (\ref{in1a_3}), we have
\begin{equation}\label{ineq1a}0\leq\intav_{\partial B_h}u(x+y)d\sigma_y-u(x) \leq \frac{1}{2}Ah^{1+\alpha},\qquad \forall x\in B_{1-\frac{1}{2}t},\ h\in(0,\textstyle\frac{1}{4}t).\end{equation}
By Lemma \ref{klm2}, $u\in \mathcal C^{1,\alpha}(B_{1-\frac{1}{2}t})$. Furthermore
\begin{equation}|u|^*_{1,\alpha;B_{1-\frac{1}{2}t}}\leq C(n,\alpha,t)(|u|_{0;B}+A).
\end{equation}
Combine (\ref{in1a_0}) and (\ref{in1a_2}), we obtain
\begin{equation}[u]_{1,\alpha;B_{1-t}}\leq C(n,\alpha,t)([\varphi]_{1,\alpha;\partial B}+|f^{\frac{1}{n}}|'_{1,\alpha;B}).\end{equation}
This completes the proof of Theorem \ref{in1a}.
\end{proof}

\vspace{0.5cm}

\section{Proof of Theorem \ref{in0a}}
The proof of Theorem \ref{in0a} is similar to the proof of Lemma 3.1 and simpler. In consideration of the integrity of this paper, we will give the whole proof. The readers can skip this section.

\medskip

We only need to consider the case $r=1$. Namely we need to prove
\begin{lem}\label{in0a'} If $u\in PSH(B)\cap\mathcal{C}(\bar{B})$ is the solution of the equation
\begin{equation}\left\{
\begin{array}{rl}
\det(u_{i\bar{j}})=f,&\textrm{in } B,\\
u=\varphi,&\textrm{on } \partial B,
\end{array}
\right.
\end{equation}
where, $\varphi\in \mathcal{C}^{0,\alpha}(\partial B)$, $f\geq 0$ and $f^{\frac{1}{n}}\in \mathcal{C}^{0,\alpha}(\bar{B})$, $\alpha\in(0,1]$. Then for any $t\in(0,1],\ x_1,x_2\in B_{1-t}$, we have
\begin{equation}|u(x_1)-u(x_2)|\leq C(n,t)([\varphi]_{0,\alpha;\partial B}+|f^{\frac{1}{n}}|'_{0,\alpha;B})|x_1-x_2|^{\alpha}.
\end{equation}
\end{lem}

\begin{proof}Like in the proof of Lemma \ref{klm1}, we consider the following functions defined on $B\times \bar B$
\begin{align*}&U(x,z)=u_x(z)=u(T_{-x}(z)),\\
&F(x,z)=f_x(z)=f(T_{-x}(z))|\det JT_{-x}(z)|^2,
\end{align*}
and the the following function defined on $B\times \partial B$
\begin{equation*}\Phi(x,z)=\varphi_x(z)=\varphi(T_{-x}(z)).\end{equation*}
For any fixed $x\in B$, we also have $u_x\in PSH(B)\cap \mathcal{C}(\bar B)$ and satisfies
\begin{equation*}\left\{
\begin{array}{rl}
\det((u_x)_{i\bar{j}})=f_x,&\textrm{in } B,\\
u_x=\varphi_x,&\textrm{on } \partial B.
\end{array}
\right.
\end{equation*}
and
\begin{equation*}u_x(0)=u(x).\end{equation*}

\medskip

Since $\varphi\in \mathcal C^{\alpha}(\partial B)$ and $f^{\frac{1}{n}}\in \mathcal C^{\alpha}(\bar B)$, $\Phi\in \mathcal C^{\alpha}(\bar B_{1-t}\times \partial B)$ and $F^{\frac{1}{n}}\in \mathcal C^{\alpha}(\bar B_{1-t}\times \bar B)$. Moreover, we have
\begin{equation}[\Phi(\cdot,z)]_{0,\alpha;B_{1-t}}\leq C(n,t)[\varphi]_{0,\alpha;\partial B},
\end{equation}
for any $z\in \partial B$, and
\begin{equation}
[F^{\frac{1}{n}}(\cdot,z)]_{0,\alpha;B_{1-t}}\leq C(n,t)|f^{\frac{1}{n}}|'_{0,\alpha;B},
\end{equation}
for any $z\in \bar B$.
Since $x_1,x_2\in \bar B_{1-t}$, we have
\begin{equation}|\varphi_{x_1}(z)-\varphi_{x_2}(z)|=|\Phi(x_1,z)-\Phi(x_2,z)|\leq C(n,t)[\varphi]_{0,\alpha;\partial B}|x_1-x_2|^{\alpha},
\end{equation}
for any $z\in \partial B$, and
\begin{equation}
|f_{x_1}^{\frac{1}{n}}(z)-f_{x_2}^{\frac{1}{n}}(z)|=|F^{\frac{1}{n}}(x_1,z)-F^{\frac{1}{n}}(x_2,z)|\leq C(n,t)|f^{\frac{1}{n}}|'_{0,\alpha;B}|x_1-x_2|^{\alpha},
\end{equation}
for any $z\in \bar B$.

\medskip

Consider $W=u_{x_1}+|x_1-x_2|^{\alpha}(-A_1+A_2(|z|^2-1))$, where
\begin{equation}A_1=C(n,t)[\varphi]_{0,\alpha;\partial B},\qquad A_2=C(n,t)|f^{\frac{1}{n}}|'_{0,\alpha;B}.\end{equation}
Then $W\in PSH(B)\cap \mathcal C(\bar B)$. On $\partial B$, we have
\begin{equation}
W=u_{x_1}-A_1|x_1-x_2|^{\alpha}\leq u_{x_2}.
\end{equation}
In $B$, we have
\begin{equation}\begin{aligned}\det(W_{i\bar j})^{\frac{1}{n}}&=\det((u_{x_1})_{i\bar j}+A_2|x_1-x_2|^{\alpha}\delta_{ij})^{\frac{1}{n}}\\
&\geq f_{x_1}^{\frac{1}{n}}+A_2|x_1-x_2|^{\alpha}\\
&\geq f_{x_2}^{\frac{1}{n}}=\det((u_{x_2})_{i\bar j})^{\frac{1}{n}}.
\end{aligned}\end{equation}
By the comparison principle, we have $W\leq u_{x_2}$. Consequently
\begin{equation}u(x_1)-u(x_2)=W(0)-u_{x_2}(0)+(A_1+A_2)|x_1-x_2|^{\alpha}\leq (A_1+A_2)|x_1-x_2|^{\alpha}.
\end{equation}
Similarly we have
\begin{equation}u(x_2)-u(x_1)\leq (A_1+A_2)|x_1-x_2|^{\alpha}.
\end{equation}
This completes the proof of Lemma \ref{in0a'}.
\end{proof}

\vspace{0.5cm}

\section{A short discussion about equations on  Hermitian manifolds}

Recently, along with the study of Hermitian manifolds, theories about the complex Monge-Amp\`ere equation on Hermitian manifolds are greatly developed. S.Dinew, S.Kolodziej and N.C.Nugyen (\cite{DK10, KN15,Ng16}) developed the thoeory of weak sotution and established $L^{\infty}$ and H\"older estimate for the solustion when the  right hand side is a nonpositive $L^p\ (p>1)$ functions; B.Guan and Q.Li (\cite{GL10}) obtained some results about the existence of smooth solutions to the Dirichlet problem with smooth data; X. Zhang and X.W.Zhang (\cite{ZhZh}) established a Bedford-Taylor $\mathcal C^{1,1}$ estimate and an interior Calabi $\mathcal C^3$ estimate; etc.

\medskip

Our $\mathcal C^{k,\alpha}\ (i=0,1,\ \alpha\in (0,1])$ estimate can also be generalized to the Hermitian case. In fact we have

\begin{thm}\label{hin0a}Let $B_r(0)$ be a ball in $\mathbb C^n$, $\omega$ is a Hermitian form on $\bar B_r(0)$. Let $i=0,1$, $\alpha\in (0,1]$, $\varphi\in \mathcal{C}^{i,\alpha}(\partial B_r(0))$ and $0\leq f^{\frac{1}{n}}\in \mathcal{C}^{i, \alpha}(\bar B_r(0))$. If $u\in PSH(B_r(0),\omega)\cap \mathcal{C}(\bar B_r(0))$ is a weak solution of
\begin{equation}\left\{\begin{array}{ll}
(\omega+\sqrt{-1}\partial\bar\partial u)^n=n!fdV, &\textrm{in }B_r(0),\\
u=\varphi, &\textrm{on }\partial B_r(0).
\end{array}\right.
\end{equation}
Then $u\in\mathcal{C}^{i,\alpha}(\bar B_r(0))$. Furher more , for any $t\in (0,1)$, we have
\begin{equation}[u]_{i,\alpha;B_{(1-t)r}(0)}\leq C(n,t)([\varphi]_{i,\alpha;\partial B_r(0)}+r^{2-i-\alpha}|f^{\frac{1}{n}}|'_{i,\alpha;B_r(0)}+r^{2-i-\alpha}|\omega|'_{i,\alpha;B_r(0)}).\end{equation}
\end{thm}

The case $i=\alpha=0$ is already known in \cite{Ng16}; the case $i=\alpha=1$ is already konwn in \cite{ZhZh}. Like in \cite{ZhZh}, one can adapt the proof of Lemma \ref{klm1} to the Hermitian case and  prove Theorem \ref{hin0a} easily. We omit the proof.

\vspace{0.5cm}

%\section*{acknowledgement}
%This paper paper is part of the first named author's PHD thesis. He would like to thank the other two authors for their suggestions and encouragements.

\end{document}